\documentclass{article}

\usepackage{hyperref}

\usepackage{amsmath,amssymb,amsfonts}
\usepackage{todonotes}
\usepackage{graphicx}
\usepackage{xcolor}
\usepackage{bm}
\usepackage{hyperref}
\hypersetup{hidelinks}
\newcommand{\Chi}{\mbox{\large$\chi$}}

\usepackage{MnSymbol}

\newtheorem{definition}{Definition}[section]

\newtheorem{theorem}{Theorem}[section]
\newtheorem{corollary}{Corollary}[theorem]

\newtheorem{proposition}[theorem]{Proposition}
\newtheorem{lemma}[theorem]{Lemma}

\usepackage{enumitem}

\newlist{level}{itemize}{4}
\setlist[level]{label={},noitemsep,topsep=0pt}
\newcounter{algo}
\renewcommand{\thealgo}{\arabic{section}.\arabic{algo}}
\newenvironment{algorithm}[1]{%
    \refstepcounter{algo}%
    \paragraph{Algorithm \thealgo}#1%
    \vspace{2pt}\hrule\vspace{5pt}%
    \begin{level}
}{%
    \end{level}%
    \vspace{5pt}\hrule\vspace{\baselineskip}%
}

\def\blackbox{\leavevmode\vrule height 5pt width 4pt depth 0pt\relax}
\newenvironment{proof}{\begin{trivlist}
\item[]\hspace{0cm}{\bf Proof:}\hskip -5pt
\hspace{0cm} }{\hfill $\blackbox$
\end{trivlist}}

\begin{document}

\title{Reducing the Gibbs effect in multimodal medical imaging by the Fake Nodes Approach}


\author{Davide Poggiali, Diego Cecchin, Stefano De Marchi}

\maketitle
\begin{abstract}
It is a common practice in multimodal medical imaging to undersample the anatomically-derived segmentation images to measure the mean activity of a co-acquired functional image. This practice avoids the resampling-related Gibbs effect that would occur in oversampling the functional image. As sides effect, waste of time and efforts are produced since the anatomical segmentation at full resolution is performed in many hours of computations or manual work.\\
In this work we explain the commonly-used resampling methods and give errors bound in the cases of continuous and discontinuous signals. Then we propose a Fake Nodes scheme for image resampling designed to reduce the Gibbs effect when oversampling the functional image. This new approach is compared to the traditional counterpart in two significant experiments, both showing that Fake Nodes resampling gives smaller errors.
\end{abstract}

\section{Introduction}
In nowadays medical imaging it is becoming more and more popular the usage of multimodal imaging. The typical setup of a multimodal imaging system allows to acquire simultaneously both anatomical and functional images of a physical body~\cite{Ehman2017, Zhang2017}, as for instance SPECT/CT\footnote{SPECT stands for \textit{Single Photon Emission Tomography}, CT for \textit{Computed Tomography}.}, PET/MRI\footnote{\textit{Positron Emission Tomography} and \textit{Magnetic Resonance Imaging}, respectively.} or PET/CT. Anatomical imaging (e.g. CT, MRI) aims to picture accurately the interiors of the body and offers a high spatial resolution, around $1 mm^3$ per voxel with the present day's machines. To the other side functional imaging (SPECT, PET in our example) aims to show physiological activity of the body under examination, offering an high sensitivity at the price of a lower spatial resolution (about $8-27 mm^3$ per voxel) than the morphological counterpart.\\
In this context, in many clinical studies researchers need to measure some statistical moments (usually the mean or median) of the functional activity inside some specific segment previously identified from the anatomical image, see e.g.~\cite{Cecchin2017}. In order to get such measure it is necessary that the anatomical and functional image have the exact same shape, which ensures that the position of all the voxels of both images correspond. There are usually two ways to achieve that:
\begin{enumerate}
    \item oversampling the functional image to the same size of the anatomical,
    \item undersampling the anatomical image and the segments to the same size of the functional image~\cite{Tustison2014, Dumitrescu2019}.

\end{enumerate}
In a recent paper~\cite{pog_gibbs21} we showed that the latter is preferable due to the Gibbs effect that occurs in oversampling. However this approach is a waste of the time and efforts that are spent in producing accurate segments at full resolution~\cite{Delgado2014}. For these reasons we have been looking for a Gibbs-free oversampling technique. \\

Inspired by the Fake Nodes Approach introduced in \cite{DeMarchi2020, DeMarchi2021}, in this work we propose a Fake Nodes-based interpolation technique that allows to oversample the functional image that takes into account the segments at full resolution, resulting in a remarkable reduction of the error due to the Gibbs effect. After a description of the interpolation scheme behind the Fake Nodes Approach, we discuss some theoretical results and some experiments.

The paper is organized as follows. After this Introduction, in Section~2 we recall the interpolation methods currently used in image resampling and give an error bound for these methods. In the Section~3 we describe the Fake Nodes interpolation and the new scheme based on Fake Nodes in multimodal image resampling. In Section~4 we describe two significant experiments to compare a commonly used resampling method with the Fake Nodes-based approach. The first experiment is {\it in silico} using an analytically-defined test image. The second one is {\it in vitro}, using a set of PET/CT scans of a phantom for with known  foreground to background ratio.\\
In the last section we conclude outlining some future research directions.

\section{Interpolation methods for image resampling}

\subsection{Image definition}
An image is given by the signal intensity at a given point in space. Mathematically we may define an {\bf image function}
\begin{equation}
    f: \Omega \subseteq \mathbb{R}^3 \longrightarrow \mathbb{R}
    \label{cimage}
\end{equation}
where the domain is a parallelepiped $\Omega = [a_1,b_1]\times [a_2,b_2]\times[a_3,b_3]$. In practice the information on an image function are the its values sampled over a regular and equispaced grid

\begin{equation}
     X = \left\lbrace (\bm x_{ijk}) = (x_i, y_j, z_k) \in \Omega \right\rbrace \;\;\;\; \text{ with }\; \;\;
\begin{array}{c}
i = 1, \ldots,  n_x\\
j = 1, \ldots, n_y \\
k = 1, \ldots, n_z
\end{array}.
\label{grid}
\end{equation}

The values are stored in the matrix $F_{ijk} = f(\bm x_{ijk})$ which is indeed $f(X)$. From now on, $F$ is referred simply as the {\bf image}. For the sake of simplicity and without loss of generality we may consider the cube $\Omega = [a,b]^3$ equally-sampled along the axes, i.e. $n_x=n_y=n_z=n$, so that $x_i = y_i = z_i = \frac{b-a}{n-1}(i-1) + a$, $i = 1, \ldots, n$.

\subsection{Interpolation for image resampling}
Resampling an image to another equispaced grid, say $X'$, is equivalent to find $f(X')$, which is in general different from $f(X)$. This can be done by interpolation. Let $\mathcal{P}f$ be the interpolant
of $f$.

\begin{definition}
Let $f$ be a $d$-variate function known on a set of $n$ distinct points $X$. Choose a set of $n$ linearly independent $B=\{b_1, \ldots, b_N\}$, we define the {\bf interpolant} over the sets $X, f(X)$, the sum
\[\mathcal{P}f = \sum_{i=1}^N c_i b_i\]
with $c_i\in\mathbb{R}$ found by imposing the conditions

\begin{equation}
\mathcal{P}f_{|X} = f_{|X}\,.
\label{interpcond}
\end{equation}

The interpolant is unique provided that
\[\det \left[b_i(x_j)\right]_{i,j = 1,\ldots,N} \neq 0.\]

\label{def_intepolation}
\end{definition}

With the notation $\mathcal{P}f_{|X} = f_{|X}$ we mean $\mathcal{P}f(\bar x) = f(\bar x)\,\,\; \forall x\in X.$

In the case of image resampling, we choose the interpolant $\mathcal{P}f$ as the (possibly unique) linear combination of a chosen set of basis functions
$$\mathcal{B} = \{B_{ijk}\}_{i,j,k = 1, \ldots, n}$$
such that
$$\mathcal{P}f(\bm x_{ijk}) = f(\bm x_{ijk})\,\,\; \forall i,j,k = 1, \ldots, n.$$
To compute the resampled image it is sufficient to evaluate the interpolant over the evaluation grid $\mathcal{P}f(X')$ and store its value in a 3-dimensional array..

In image resampling, since the number of voxels $n^3$ can be extremely large, it is necessary to make some assumptions on the basis functions to compute the resampling image in a relatively short time.
Accordingly with ~\cite{burger_burge_2009, Getreuer2011, pog_gibbs21}, the basis should satisfy the following characteristics.
\begin{itemize}
\item[I)] {\bf Separable}: given a function $b$ the basis $B_{ijk}$ is then the product of translates of $b$ at the samples, that is
\begin{equation}
    B_{ijk}(x,y,z) := b(x-x_i)\, b(y-y_j) \,b(z-z_k),
\label{separable}
\end{equation}
    $x\in\Omega,\; \; i,j,k = 1, \ldots, n.$
\item[II)] {\bf Cardinal}:
    \begin{equation}
  B_{ijk}(\bm x_{i'j'k'})= \left\lbrace\begin{array}{lc}
  1 & \text{if } (i,j,k)=(i',j',k')\\
  0 & \text{else}
  \end{array}\right.
    \end{equation}

\item[III)] {\bf Normalized}:
\begin{equation}
    \sum_{ijk=1}^n B_{ijk}(\bm x ) = 1 \;\;\; \forall \bm x\in [a,b]^3.
\end{equation}

\item[IV)] {\bf Compact support}:
\begin{equation}
\exists \alpha \in \mathbb{R}^{+} \text{ s.t. }
    B_{ijk}(\bm x) = 0 \;\;\forall \|\bm x  -\bm x _{ijk}\|_{\infty}\ge \alpha.
\end{equation}

\end{itemize}
The reader can find some examples of basis functions satisfying all these characteristics in~\cite{burger_burge_2009, pog_gibbs21}. For instance the so-called {\it trilinear interpolation} corresponds to choose the function $b$ as
\[b(x) = \left\lbrace \begin{array}{cc}
    1 - \frac{|x|}{h} & \text{ if } |x|\leq h \\
   0  & \text{otherwise}
\end{array} \right.\]
with $h = \frac{b-a}{n}$ the distance between two samples, and to define the basis functions $B_{ijk}$ by consequence as in Eq.~\eqref{separable}.\\
Under the hypothesis I) we can write the interpolant function as

\begin{equation}
    \mathcal{P} f (\bm x) = \sum_{ijk = 1}^ n f(\bm x_{ijk}) B_{ijk}(\bm x).
\label{interpolant}
\end{equation}

Such interpolant can be quickly computed as the result of three subsequent tensor products as the following holds for hypothesis II)
\begin{equation}
    \mathcal{P} f (x, y, z) = \sum_{k=1}^n \sum_{j = 1}^n \sum_{i = 1} ^ n F_{ijk}  b(x-x_i)\, b(y-y_j) \,b(z-z_k)
\end{equation}

or, in Einstein's summation convention
\[ \mathcal{P} f (x, y, z) =  F_{ijk}  b(x-x_i)\, b(y-y_j) \,b(z-z_k). \]

It has to be noticed that such formulation of image resampling and the theoretical results presented in this work can be easily extended to any dimension.

In particular, it is possible to speed up the computation of the interpolant, as it is shown by the following Lemma, whose proof is given in~\cite{pog_gibbs21}.

\begin{lemma}
If $\mathcal{P}f$ is an interpolant of $f$ satisfying I), II) and IV) then, for every evaluation point $(x,y,z) \in [x_{p},x_{p+1}]\times[y_q, y_{q+1}]\times[z_{r}, z_{r+1}]$ it exists a strictly positive integer $a\in\mathbb{N}$ such that

\begin{equation}
    \mathcal{P} f (x, y, z) =  \sum_{k=max(1, r-a)}^{min(n,r+a-1)}\sum_{j=max(1, q-a)}^{min(n,q+a-1)} \sum_{i=max(1, p-a)}^{min(n,p+a-1)} F_{ijk}  b(x-x_i)\, b(y-y_j) \,b(z-z_k).
\end{equation}
\label{lemma1}
\end{lemma}

This means that the local compactness of the support of the basis functions ensures that in each dimension only the $a$ nodes before and after the evaluation point have an active role in evaluating the interpolant.

\subsection{Error bound for image resampling}

In this section we give an error bound for the image interpolation which relies on the {\it modulus of continuity} of a function~\cite{Bugajewski2015, Aronszajn1956, Costarelli2018, Costarelli2019}.

\begin{definition}
Let $f: \Omega \subseteq \mathbb{R}^d \longrightarrow \mathbb{R}$ be a piecewise continuous function. The function $f$ admits {\bf a modulus of continuity}  $\omega_{\bm x_0}(\delta): \mathbb{R}^{+} \longrightarrow \mathbb{R}^{+}$
in $\bm x_0\in\Omega$ if
\[ |f(\bm x) - f(\bm x_0) | \leq \omega_{\bm x_0}(\delta) \;\;\forall \bm x \in \Omega \text{ s.t. } \|\bm x- \bm x_0\|\leq \delta. \]
where $\|\cdot\|$ is an arbitrary norm.
\end{definition}
The definition of modulus of continuity given here is slightly relaxed from the one given in literature in order to include piecewise continuous functions. Namely, in our formulation it is not necessarily true that
\[\lim_{\delta\to 0}\omega_{\bm x_0}(\delta) = 0.\]

We can extend this definition and consider the {\bf global modulus of continuity}, say $\omega(\delta)$,
\[ |f(\bm x) - f(\bm y) | \leq \omega(\delta) \;\;\forall \bm x, \bm y \in \Omega \text{ s.t. } \|\bm x- \bm y\|\leq \delta. \]
It is noteworthy to recall that if a function admits a local modulus of continuity for every $\bm x_0\in \Omega$, then admits a global modulus of continuity
$$\omega(\delta) :=\sup _{\bm x_0 \in \Omega} \omega_{\bm x_0}(\delta).$$

We now enounce the definition of discontinuity jumps, that will be useful for quantifying the discontinuities of a picewise continuous signal.
\begin{definition}
Let $f:\Omega \subseteq\mathbb{R}^3 \longrightarrow \mathbb{R}$ picewise continuous on some pairwise disjoint subdomains
\[ \Omega = \overset{\circ}{\bigcup}_{i=1}^m  \Omega_i\]
the {\bf jump} between two subdomains $\Omega_i,\Omega_j, \;i,j = 1, \ldots,m$ is the positive real number defined as

\[ D(\Omega_i,\Omega_j) := \sup \left(J_{ij}  \bigcup J_{ji}\right)\]

where
\[J_{ij} = \left\lbrace \lim_{\bm x\to \bm y}|f(\bm x) - f(\bm y)| \;,\;\;\bm x\in \Omega_i, \bm y\in \partial_i\Omega_j  \right\rbrace\]
with $\partial_i\Omega_j = \partial \Omega_i \bigcap \Omega_j$ is the part of the border of $\Omega_i$ that belongs to $\Omega_j$.

In the special case that $\partial\Omega_i \cap \partial\Omega_j = \emptyset $, and by consequence also $J_{ij}$ and $J_{ji}$ are empty sets, the jump is considered equal to zero.
\label{jump}
\end{definition}

We now enounce two useful Lemmas for the estimation of the error at a given evaluation point.

\begin{lemma}
A $d$-variate, Lipschitz-continuous function $f$ admits a modulus of continuity of the form
\[\omega(\delta) = K\delta \]
being $K$ the Lipschitz constant.
\label{lemma_cont}
\end{lemma}
The proof follows trivially from the definitions of Lipschitz continuity and global modulus of continuity.

\begin{lemma}
A picewise $d$-variate, Lipschitz-continuous function $f$ admits a modulus of continuity of the form
\[\omega(\delta) = K\delta + D \]

\label{lemma_disc}
\end{lemma}

\begin{proof}
Suppose for simplicity that the function domain is divided in only two disjoint subdomains $\Omega = \Omega_1 \;\overset{\circ}{\bigcup}\;
\Omega_2$ in both of whom $f$ is Lipschitz-continuous with respective constants $k_1, k_2$.

Consider now two generic points $x,y\in\Omega$. If $x$ and $y$ both belong to the same subdomain the thesis follows by Lemma~\ref{lemma_cont}.

Otherwise, once recalled the symbols used in Definition~\ref{jump}, for any choice of $\bm\xi_1 \in \Omega_1$ and $\bm\xi_2 \in \partial_1\Omega_2$ we get
\[ |f(\bm x) - f(\bm y)| \leq  |f(\bm x) - f(\bm \xi_1)| + |f(\bm \xi_1) - f(\bm \xi_2)| + |f(\bm \xi_2) - f(\bm y)|.\]
This inequality still true for the limit
\[|f(\bm x) - f(\bm y)|\leq \lim_{\xi_1 \to\xi_2} \left\{ |f(\bm x) - f(\bm \xi_1)| + |f(\bm \xi_1) - f(\bm \xi_2)| + |f(\bm \xi_2) - f(\bm y)|\right\} \leq\]
\[\leq k_1\delta + D(\Omega_1, \Omega_2) + k_2\delta.\]

The thesis follows with $D = D(\Omega_1, \Omega_2)$ and $K = k_1 + k_2$, as  $\forall \bm x, \bm y \in \Omega \text{ s.t. } \|\bm x- \bm y\|\leq \delta$. This inequality holds also in case $\partial_1\Omega_2$ is an empty set, in which case we consider $\bm\xi_1 \in \partial_2\Omega_1$, $\bm\xi_2 \in\Omega_2$ and the limit for $\xi_2\to\xi_1$.

In the case of $m>2$ subdomains with  $\partial\Omega_i \cap \partial\Omega_j = \emptyset$, the triangular inequality is repeated for all the subdomains intersecting the straight line between $\bm x \in \Omega_i$ and $\bm y \in \Omega_j$. In the worst case such line touches all the $m$ subdomains hence
\[K = m\;\max_{i=1, \ldots, m} k_i\]
with $k_i, \; i= 1,\ldots,m$ the Lipschitz constants of each subdomain, and
\[D = m\;\max_{i,j =1, \ldots, m} D(\Omega_i,\Omega_j).\]
\end{proof}

We now prove an upper bound for the pointwise interpolation error given by any modulus of continuity.

\begin{theorem}
Let $f$ be a trivariate and bounded function admitting a local modulus of continuity $\omega_{\bm x} (\delta)$. Let $\mathcal{P}f$ be its interpolant built according to I), II), III) and IV). Then, there exits $\delta^* > 0$ such that
\[ |\mathcal{P}f (\bm x) - f(\bm x)| \leq \omega_{\bm x} (\delta^*)\,. \]

\label{error}
\end{theorem}

\begin{proof}
From Eq.~\eqref{interpolant} and from III) we get
\begin{eqnarray*}
\mathcal{P}f (\bm x) - f(\bm x)& =& \sum_{ijk=1}^n f(\bm x_{ijk}) B_{ijk}(\bm x) - f(\bm x)  \left(\sum_{ijk=1}^n B_{ijk}(\bm x) \right)\\
&=& \sum_{ijk=1}^n \left( f(\bm x_{ijk}) - f(\bm x)\right) B_{ijk}(\bm x).
\end{eqnarray*}
We know from Lemma~\ref{lemma1} that only some nodes in a neighbourhood of $\bm x$ have to be taken into account in the interpolation formula. To be precise the $\bm x_{ijk}$ such that $\|\bm x - \bm x_{ijk}\|_{\infty} < \alpha$. Since $f$ admits a modulus of continuity we can state
\[ \left| f(\bm x_{ijk}) - f(\bm x)\right| \leq \omega_{\bm x} (\delta^*),  \]
with $\delta^* = \alpha$ and $i,j,k$ in the range indicated in Lemma~\ref{lemma1}. By using the normalized basis
\[ \mathcal{P}f (\bm x) - f(\bm x) \leq \sum_{ijk=1}^n \omega_{\bm x} (\delta^*) B_{ijk}(\bm x) = \omega_{\bm x} (\delta^*).
   \]
and similarly
\[ \mathcal{P}f (\bm x) - f(\bm x) \geq - \omega_{\bm x} (\delta^*)\,,
   \]
which proves the Theorem.
\end{proof}

In a similar way we can find a global upper bound to the interpolation error, simply by taking the supremum of the local continuity moduli.

\begin{corollary}
Let $f$ be a trivariate and bounded function admitting a global modulus of continuity $\omega (\delta)$. Let $\mathcal{P}f$ be its interpolant built according to I), II), III) and IV). Then, there exits $\delta^* > 0$ such that
\[ |\mathcal{P}f (\bm x) - f(\bm x)| \leq \omega (\delta^*)\,. \]
\end{corollary}

\subsection{Gibbs effect in image oversampling}
We usually refer to Gibbs effect as the oscillations around a discontinuity that occurs when approximating a signal with a truncated Fourier series~\cite{Jerri1998, Fornberg2011}. Such effect is present also in image (re)sampling, often called ``ringing effect'' for the wavelike, concentric oscillations that appears around any sudden change of signal intensity\cite{Lehmann1999, chloa20}. Unlike the {\it Runge effect}, the magnitude of the oscillation goes to a plateau as the number of interpolation nodes goes to infinity. The presence of the Gibbs effect in multimodal medical imaging has already been discussed in~\cite{pog_gibbs21} where a "natural" error analysis has been presented. Similar results can be obtained by the results of the previous section. In fact, we can build the basis functions so that their support is $\alpha \propto \frac{1}{n}$. Hence, the interpolation error bound given in Theorem~\ref{error} goes to zero as $n\to\infty$ if the image function $f$ is Lipschitz-continuous at the evaluation point $\bm x$. By  Lemma~\ref{lemma_cont} indeed the image function admits a continuity modulus $\omega_{\bm x}(\delta) \leq K\delta$ that has the role of an error bound once evaluated on $\alpha$, and goes to zero as $n\to\infty$. On the other hand, if $f$ is only piecewise Lipschitz-continuous and the evaluation point belongs to the border of a subdomain $\Omega_i$, then by Lemma~\ref{lemma_disc} the image function admits a continuity modulus $\omega_{\bm x}(\delta) \leq K\delta + D$ that evaluated on $\alpha$ is an error bound, which converges to $D$ as $n$ increases to infinity.\\
In the next section we introduce the Fake Nodes interpolation, aiming to interpolating a discontinuous image on a mapped set of nodes so that it appears as continuous, getting rid of the Gibbs effect.

\section{Fake Nodes interpolation}
Fake Nodes is an interpolation paradigm introduced in~\cite{DeMarchi2020, DeMarchi2021} that can reduce the Gibbs effect in different frameworks and applications~\cite{DeMarchi2020a,graspa, betagamma}.
\subsection{Definition of Fake Nodes interpolation}
The Fake Nodes interpolant can be defined as in  Def.~\ref{def_intepolation}, with the difference that the interpolation and evaluation nodes are subject to mapping. For reasons that will be clarified soon, in this paper we choose a slightly different definition of Fake Nodes with respect to the above quoted literature.

\begin{definition}
Let $f$ be a $d$-variate function sampled on a set of $n$ distinct nodes $X$. Let the set of functions $\{b_1, \ldots, b_m\}$, $m\geq n$ be a basis. Consider the map
$S: \Theta \subseteq \mathbb{R}^d \to \mathbb{R}^d,$
and $\mathcal{P}f$ the interpolant over $X^S \supseteq S(X)$ to $Y \supseteq f(X)$, $X^S, Y$ both of cardinality $m$. We call {\bf Fake Nodes interpolant} the function
\[\mathcal{R}^Sf := \mathcal{P}f \circ S .\]
\label{def_FNintepolation}
\end{definition}

The following proposition holds.

\begin{proposition}
The Fake nodes interpolant as defined in Def. ~\ref{def_FNintepolation} satisfies the interpolation conditions~\eqref{interpcond}, i.e $\mathcal{R}^Sf_{|X} = f_{|X}$.
\end{proposition}
\begin{proof}
Being $\mathcal{R}^Sf = \mathcal{P}f \circ S$ and the fact that $\mathcal{P}f$ is the interpolant over the whole supersets $X^S, Y$, in particular we get
\[ \mathcal{P}f(S(\bar x)) = f(\bar x) \;\;\forall \bar x \in X, \]
that is
\[\mathcal{R}^Sf(\bar x) = f(\bar x) \;\;\forall \bar x \in X.\]
\end{proof}

\noindent {\bf Remark}. This definition of Fake Nodes differs to classical ones since considers a larger number of basis functions and interpolation nodes. To use the image resampling defined in the previous section, we must have to consider a regular, equispaced grid as input. Since the mapped nodes $S(X)$ is not equispaced in general, we choose to interpolate over a larger grid by adding some nodes in order to obtain a regular and equispaced grid $X^S\supseteq S(X)$.
In correspondence of the newly added nodes, we can choose some values to form the set $Y \supseteq f(X)$. Since we want to avoid the Gibbs effect, which occurs next to a discontinuity, we shall choose such values in order to get a smooth, regular image. Given these choices, it becomes mandatory to use a larger basis of functions.


\subsection{Fake Nodes approach for multimodal image resampling}
As discussed in the Introduction, in multimodal imaging we have a high resolution morphological image $I$ along with its segmentation $M$, composed by integer values indicating the number of segment each voxel belongs. We also have a low resolution functional image $F$, and we aim to estimate the mean value of the functional image inside each segment with the lowest possible error.

The image $I$ is defined on a grid $X_{high}$ of the domain $\Omega = [a,b]^3$. The segmentation is the set of pairwise disjoint subsets of the domain $\Omega$ (cf. e. g. ~\cite{Pham2000}), that is
\[ \Omega = \bigcup_{k=0}^{m} \Omega_k  \;\;\text{ with }\;\; \Omega_i \cap \Omega_j = \varnothing\;\;\forall i,j = 1,\dots,m,\]
and is known up to the resolution of $I$, and stored in the matrix
$$M = \sum_{k=0}^m k\; \Chi_{\Omega_k}(\bm{x}_{ijk})\;\text{ with }\bm x_{ijk} \in X_{high}$$
of the same size of $I$. By convection, the zero-indexed segment represent the background.\\
On the other hand, the functional image $F$ is defined on another, smaller grid of $[a,b]^3$, say $X_{low}$.

We then propose to oversample the functional image $F$ to the dimension of $I$ and $M$ by using a Fake Nodes interpolation with mapping function
\[ S(\bm{x}) = \bm{x} +  \sum_{k=0}^m \bm{\alpha}_k \Chi_{\Omega_k}(\bm{x}), \]
where $\bm{\alpha}_k = ((b-a)k, 0, 0)$.
In other words, we create as many images as the segments, each with the values of the functional image inside the relative segment and zero elsewhere. These images are then stacked over the first axis, creating a long, unique 3-dimensional image.

As stated in the remark, this would be not sufficient to overcome Gibbs effect, as the so obtained signal has still some sudden change of intensity. Since we have the freedom to choose the values on the new nodes, we apply Gaussian blurring on the image and re-impose the original intensity values on the set $S(X_{low})$. This process can be repeated a number of times, until a smooth and continuous-like image is obtained. In the experiments below, we applied this strategy of blurring-reimposing three times. The result of this process can be seen in Fig.~\ref{fig:fakenodes}.\\
The so-obtained image is then oversampled to the same resolution of $I$ and $M$, and then recomposed according to the mapping $S$, as indicated in Def.~\ref{def_FNintepolation}. This method is explained in better detail in the following Algorithm~\ref{FakeNodesAlg}.\\
\begin{algorithm}{Fake Nodes image resampling}
\item \textbf{Inputs:} \\
$I$ the target, high-resolution image;\\
$M$ the segmentation of $I$ in $m+1$ segments;\\
$F$ the input, low-resolution image;
$\mathcal{G}$ a Gaussian function of fixed variance.
\item \textbf{Execution:}
\begin{level} \item
1. Downsample the segmentation image $M$ to the resolution of $F$, call the result $M_{low}$;\\
2. {\bf for} $segm$ in $0$:$m$:
\begin{level}\item
2.1. Impose the values of the current segment to the image at $segm$ level and add the result to a list of images $tempF$:\\ \hspace*{1cm} $tempF[segm] = F[M_{low}==segm]$\\
2.2. {\bf repeat 3 times:}
\begin{level}\item
2.2.1. Gaussian smoothing by convolution:\\ \hspace*{1cm} $tempF[segm] = tempF[segm] \ast \mathcal{G}$\\
2.2.2. Re-imposition: $tempF[segm] = F[M_{low}==segm]$
\end{level}
\end{level}
\end{level}
3. Stack the images $tempF[]$ along the first axis, call the result $fakeF$;\\
4. Oversample the image $fakeF$ to the resolution of $I$ with a basis of choice, call the result $HRfakeF$;\\
5. Initialize the result image $highF$ as a matrix with the size of $I$;\\
6. {\bf for} $segm $in $0$:$m$:
\begin{level}\item
6.1. Impose the values at each segment:\\
\hspace*{1cm} $highF[M==segm] =HRfakeF[M==segm] $
\end{level}

\item \textbf{Output:} \\
$highF$, the resampled version of $F$ using Fake Nodes.
\label{FakeNodesAlg}
\end{algorithm}
{\bf Remark}. The whole algorithm has a computational cost of $m$ times the computational cost of the corresponding resampling algorithm, where $m$ is the number of segments. This means that if the number of segments is too large it would be better to reduce it by merging some segments with similar values at the boundaries.

\section{Experiments and results}
For testing the Fake Nodes resampling and compare it to the usual resampling methods, we carried on two experiments. The first one is made on the Shepp-Logan phantom~\cite{Shepp1974}, an analytically-defined image with constant intensity in each of its ellipsoidal segments. Since the exact value is known, in this case we can compute the resampling errors. An example of Shepp-Logan phantom is reported in Fig.~\ref{fig:shepplogan}.\\

\begin{figure}[!hbt]
    \centering
    \includegraphics[width=\linewidth]{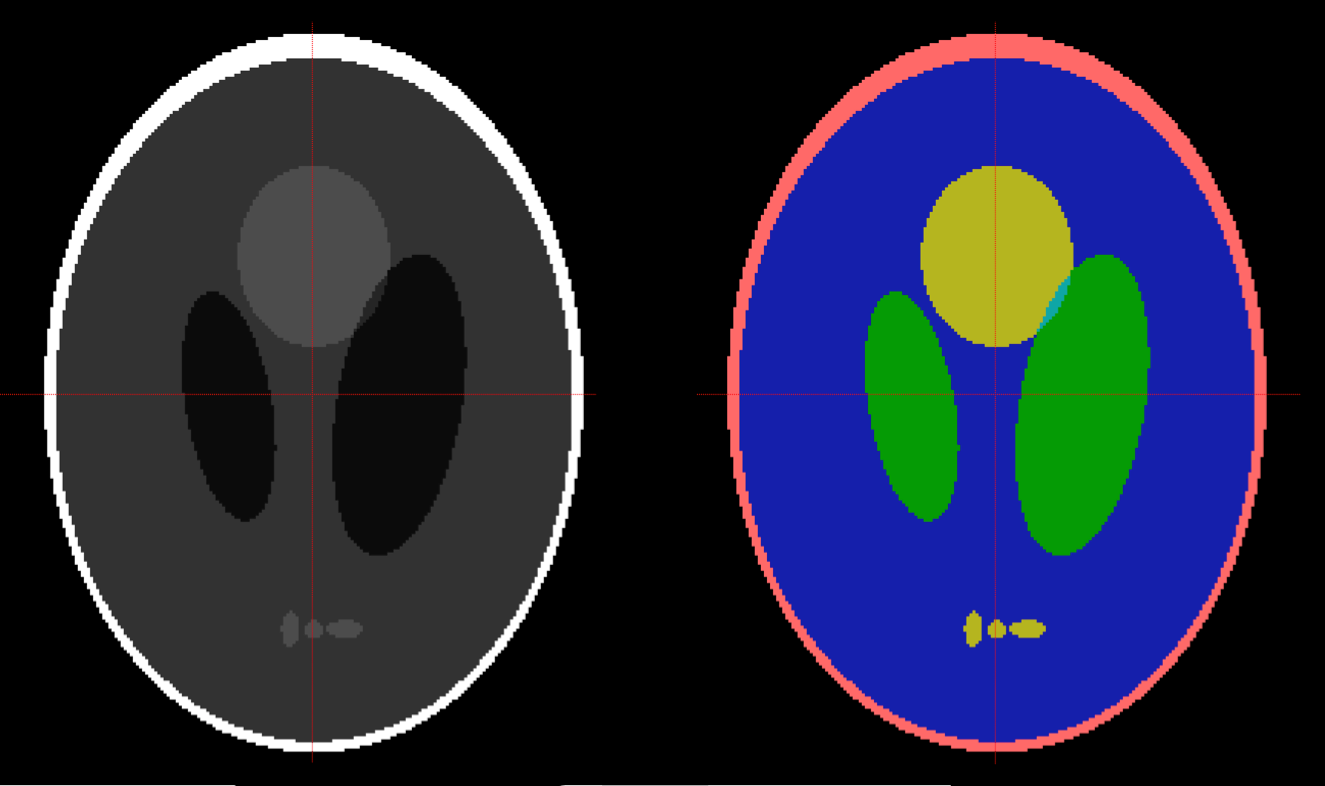}
    \caption{Left: an axial slice of the Shepp-Logan phantom. Right: segmentation of the same slice, with the segments colored as follows: black segment number 0, red 1, green 2, blue 3, yellow 4, cyan 5.}
    \label{fig:shepplogan}
\end{figure}

The second test is from a real PET/CT dataset of scans of the same physical phantom. This phantom is a closed container made of plastic. Inside the containers six spheres of diameters $10, 13, 17, 22, 28$ and $37\,mm$  are filled with the tracer of known activity (we call `hot' or foreground the interior of the spheres), whereas the remaining part of the phantom (here called `cold' or background) is filled with water and tracer of known activity.

In both experiments we compared the results obtained by oversampling the functional image using the trilinear interpolation basis, cfr.~\cite{burger_burge_2009} in combination with the Fake Nodes scheme explained in the previous paragraph. In both cases the resampling is performed in ANTs~\cite{AVANTS2008, Avants2011}, a popular software suite in neuroimaging written on top of ITK~\cite{itk_2014}. The Fake Nodes scheme is implemented in a set of Python scripts. All the scripts written for this paper and some sample data can be found in the repository of this paper \url{https://github.com/pog87/FakeResampling3D}.\\
Another package that can be useful to understand image resizing can be found at~\cite{ResizeRight}.

\subsection{Experiment 1: Shepp-Logan phantom}

We generated two Shepp-Logan phantom images: an high resolution image of size $256^3$ and a low resolution image of $128^3$ voxels with the Python packages {\tt tomopy} and {\tt nibabel}~\cite{Brett2020}.
Since the Shepp-Logan phantom is picewise constant, segmentation has been performed simply by identifying the voxels of the same values. The phantom and its segmentation can be observed in Fig.~\ref{fig:shepplogan}. The low resolution `functional' image was resampled to the high resolution using Trilinear interpolation and Fake Nodes Trilinear interpolation. The Fake Image generated for resampling can be observed in Fig.~\ref{fig:fakenodes}. At last, the mean value per segment has been computed and compared with the actual, constant value of the segment. The results can be observed in Fig.~\ref{fig:fakeresults} and Table~\ref{tab_shepp}. As can be easily seen, the Fake Nodes resampling leads to more accurate mean value for every segment, including the background.

\begin{figure}[!hbt]
    \centering
    \includegraphics[width=\linewidth]{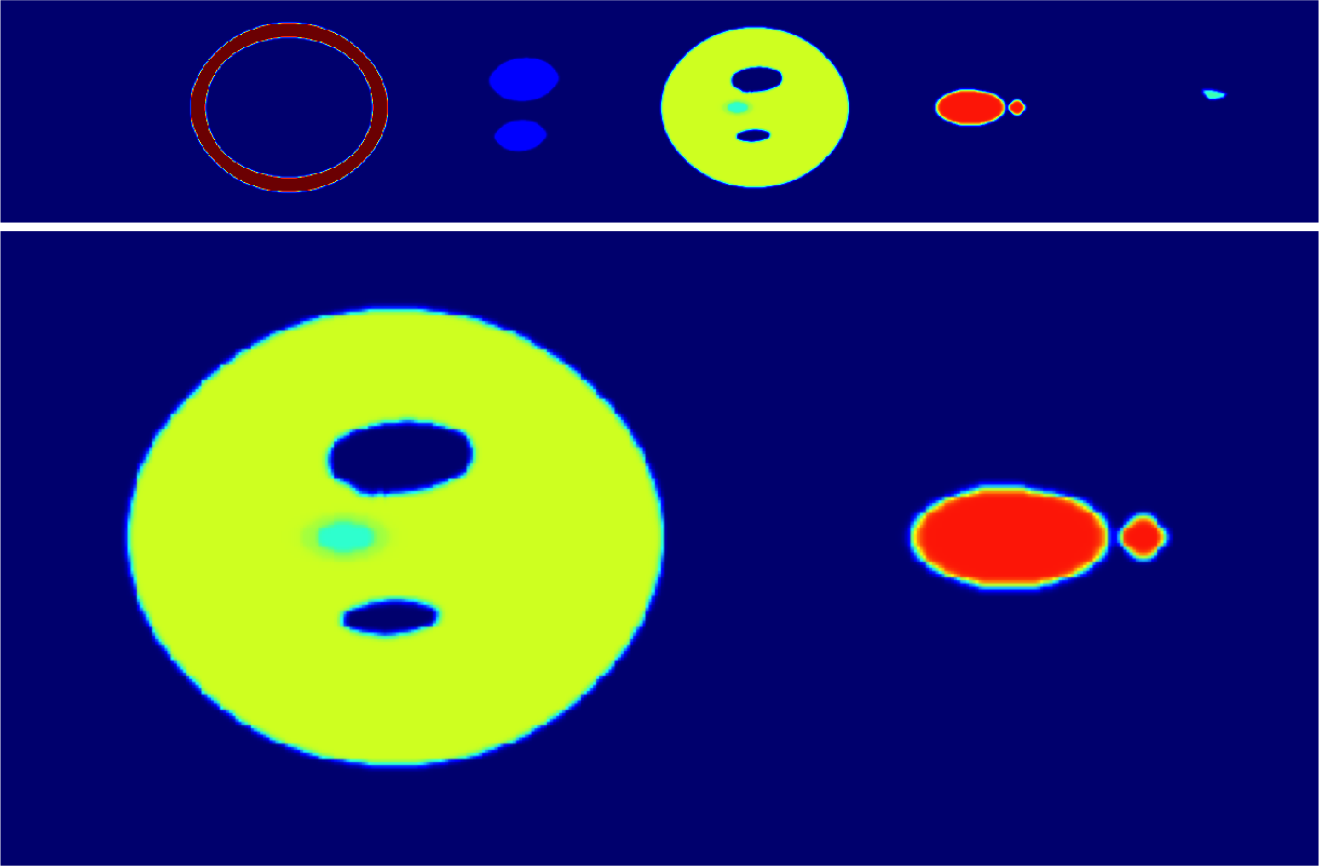}
    \caption{Top: a Fake Image to be resampled, obtained from the low resolution image with Algorithm~\ref{FakeNodesAlg}. Bottom: a zoom of the same, showing that the signal is smooth. The jet colormap (blue-green-yellow-red) has been used to better represent the smoothness of the signal.}
    \label{fig:fakenodes}
\end{figure}

\begin{figure}[!hbt]
    \centering
    \includegraphics[width=\linewidth]{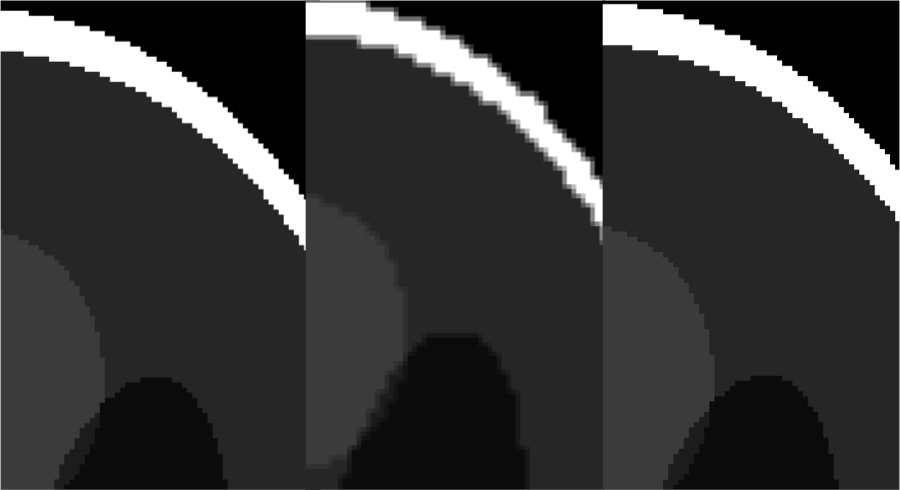}
    \caption{Left: a zoom of an axial slice of the high resolution Shepp-Logan phantom. Center: the same slice oversampled with Trilinear interpolation. Right: the same slice oversampled with Trilinear interpolation and Fake Nodes.}
    \label{fig:fakeresults}
\end{figure}

\begin{table}[!h]
\begin{tabular}{|c||l||l l||l l|}
\hline
 Segment &   Ref. &  Trilinear && Fake-Trilinear &\\ \cline{3-6}
   &  &  mean value  &  abs. error   &   mean value  &  abs. error   \\ \hline
 0 & 0.  &    $1.979\cdot 10^{-3}$  &     $1.979\cdot 10^{-3}$  &    $0.986\cdot 10^{-16}$  &    $0.986\cdot 10^{-16}$  \\
 1 &1.   & 0.823  &  0.177  &  1.000  &     $0.9333\cdot10^{-4}$  \\
2  &  0.05  &   $0.569 \cdot 10^{-1}$  &$0.686\cdot10^{-2}$  &     $0.500\cdot 10^{-1}$  &     $1.379\cdot 10^{-6}$  \\
3  &0.2  & 0.216  &     $0.157 \cdot 10^{-1}$ &$0.200$  &$0.204\cdot10^{-5}$  \\
4  &0.3  & 0.296  &$0.390\cdot10^{-2}$  & $0.300$ & $0.797\cdot10^{-5}$ \\
5   &0.15 &    0.160   &   $0.995\cdot10^{-2}$  &   $1.499\cdot10^{-1}$   &   $0.608\cdot10^{-4}$  \\ \hline
\end{tabular}

\label{tab_shepp}
\caption{Segment index, Reference value, mean value and absolute error per segment relative to the oversamplings of the Shepp-Logan phantom with Trilinear interpolation and Trilinear interpolation with Fake Nodes.}
\end{table}

\subsection{Experiment 2: PET/CT phantom data}
The RIDER PET/CT Phantom dataset~\cite{rider_phantom} has been downloaded from \hyperlink{https://wiki.cancerimagingarchive.net/display/Public/RIDER+Phantom+PET-CT}{the dataset official webpage} hosted on the Cancer Imaging Archive (TCIA)~\cite{Clark2013} website. This dataset consists in 20 repeated scans of the RIDER PET/CT Phantom filled with Ge68 tracer and reconstructed using 3D Filtered Back Projection (FBP)~\cite{Jha2014}. We know that in principle  the foreground-to-background ratio should be exactly equal to 4, as the interior of the spheres has been filled with a radioactive tracer of activity four times higher than the water lying outside them.\\ The CT scan has been automatically segmented using k-means clustering in four segments: background, phantom plastic shell, `cold' water and `hot' spheres. A slice of a phantom scan can be found in~\ref{fig:rider}. The foreground to background ratios have been computed as
\[ FBr = \frac{PET[hot]}{PET[cold]}\]
for both the Trilinear and Fake Nodes Trilinear interpolation methods. In both cases a Partial Volume Correction (PVC) method~\cite{Muller1992} was applied to resampled PET before computing the mean value per segment, with an estimated Full Width at Half Maximum (FWHM)~\cite{Cecchin2015} of $6mm$.\\
The results of this experiment can be observed in the Raincloud plot~\cite{Allen2019} at Fig.~\ref{fig:riderres}. The Fake Nodes approach results in a mean error reduced of about $14\%$ with respect to the usual Trilinear resampling. The t-test between the two groups gives a p-value of about $ p \approx 0.7\cdot 10^{-8}$, indicating that the difference between the two groups is significant.

\begin{figure}[!hbt]
    \centering
    \includegraphics[width=\linewidth]{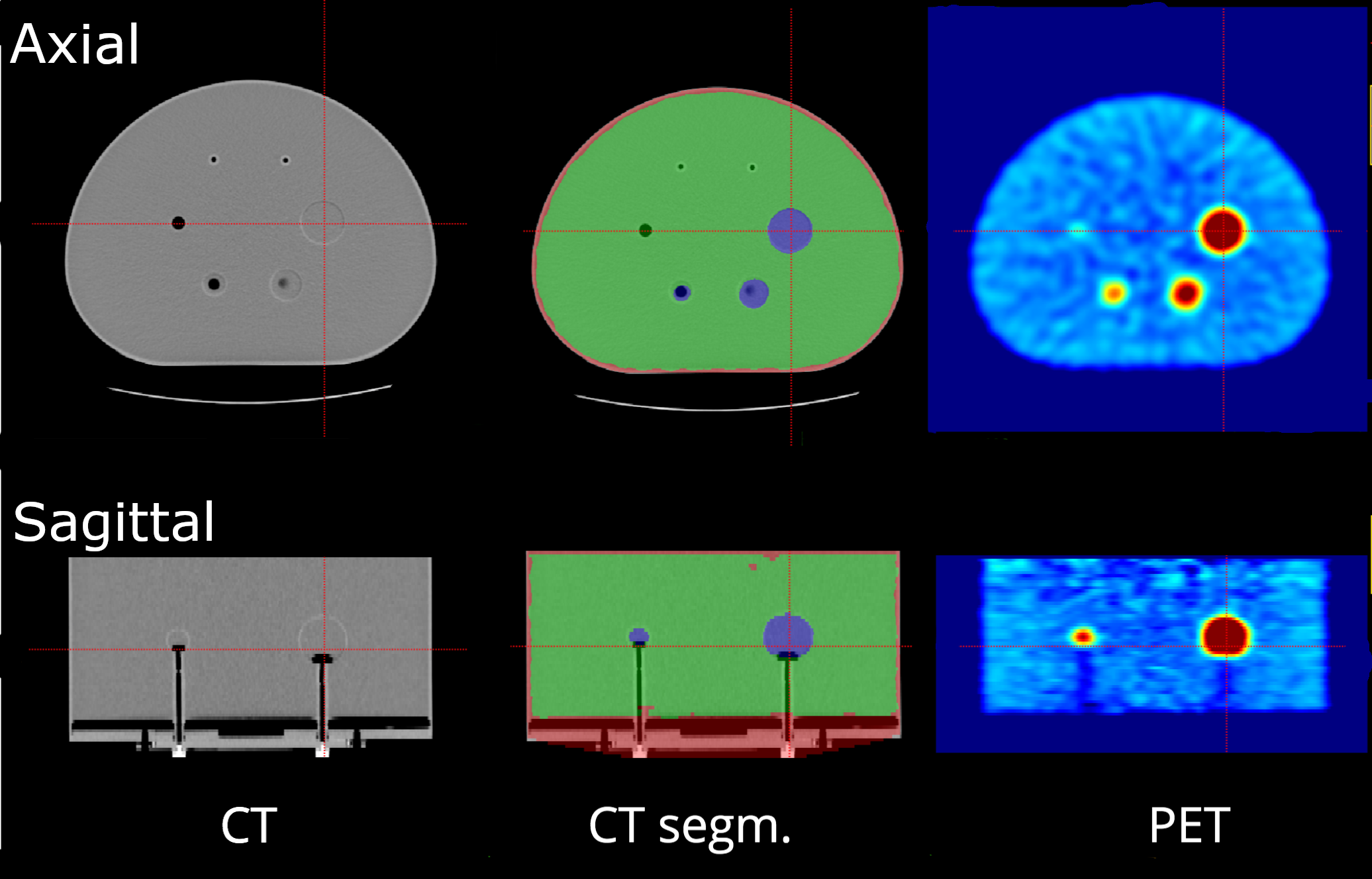}
    \caption{The same axial (top) and sagittal (bottom) slices of the the PET RIDER phantom CT (left), CT segmentation (center) and PET scan (right). The segmentation indexes are: background (black) 0, phantom shell (red) 1, `cold' water (green) 2, `hot' spheres (violet) 3. The jet colormap has been used for PET images}
    \label{fig:rider}
\end{figure}

\begin{figure}[!hbt]
    \centering
    \includegraphics[width=\linewidth]{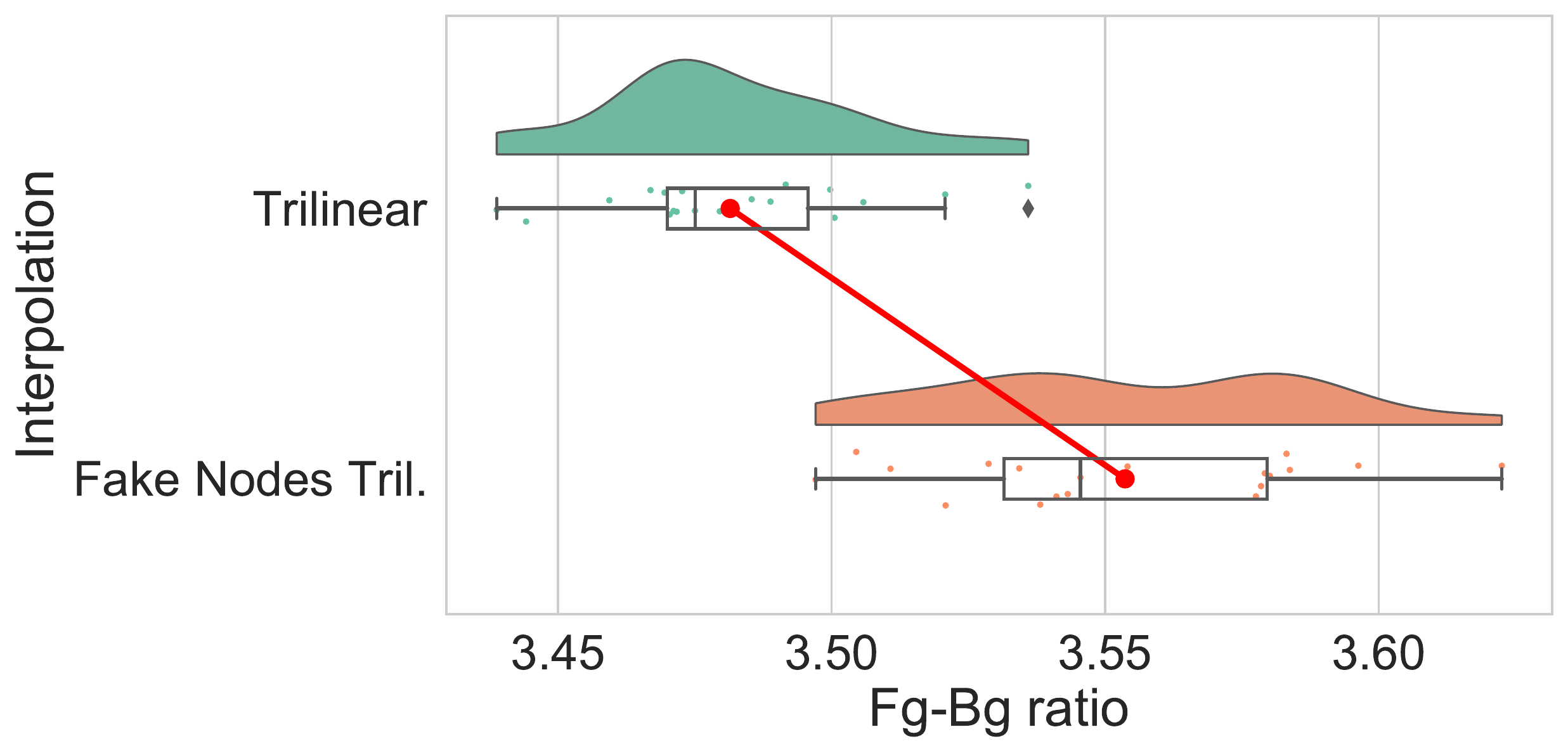}
    \caption{Results of the RIDER phantom experiments, pictured as Raincloud plots. Comparison between foreground to background ratios obtained using Trilinear interpolation (top, green) and Fake Nodes Trilinear interpolation (bottom, orange).}
    \label{fig:riderres}
\end{figure}

\subsection{Comments to the results}
The first experiment shows that resampling a functional image with the Fake Nodes approach results in an almost perfect set of mean values for each segment, with an impressive improvement with respect to the traditional image resampling. In a less controlled and more realistic situation, the use of Fake Nodes in resampling does lead to a significant improvement in the ratios, but still the result appears not close enough to the target value of 4. These results are coherent with the dataset documentation~\cite{rider_phantom}, indicating that an exact quantification of the tracer concentration in PET is still far from being achieved.

\section{Conclusions and future work}

In this paper we have proposed a Fake Nodes scheme for medical image oversampling that can be applied to any image of arbitrary dimension. In Sections 2 and 3, we introduced image resampling techniques as interpolatory methods and gave an error bound which is different in the cases of continuous and piecewise continuous signals. Fake Nodes approach has proved to bring a more accurate functional quantification in the experiments, both using an analytically-defined phantom and real data from the scans of a PET/CT physical phantom.\\
The results on the physical phantom appeared to be still too far from the target. This may indicate that the Partial Volume Correction, Attenuation Correction and a Fake Nodes-like oversampling should be all applied during the reconstruction phase in order to achieve an exact quantification in PET imaging.

\section{Acknowledgments}
This research has been funded by the PNC - Padova Neuroscience Center, University of Padova (Italy) as part of the project ``A computational tool for neurodegenerative stratification using PET/RM''.\\
This research has been accomplished within Rete ITaliana di Approssimazione (RITA) and partially funded by GNCS-IN$\delta$AM.

\bibliography{mybibfile}

\begin{thebibliography}{10}
\expandafter\ifx\csname url\endcsname\relax
  \def\url#1{\texttt{#1}}\fi
\expandafter\ifx\csname urlprefix\endcsname\relax\def\urlprefix{URL }\fi
\expandafter\ifx\csname href\endcsname\relax
  \def\href#1#2{#2} \def\path#1{#1}\fi

\bibitem{Ehman2017}
E.~C. Ehman, G.~B. Johnson, J.~E. Villanueva-Meyer, S.~Cha, A.~P. Leynes,
  P.~E.~Z. Larson, T.~A. Hope,
  \href{https://doi.org/10.1002/jmri.25711}{{PET/MRI:} {Where} might it replace
  {PET/CT?}}, J. Magn. Reson. Imaging 46~(5) (2017) 1247--1262.
\newblock \href {https://doi.org/10.1002/jmri.25711}
  {\path{doi:10.1002/jmri.25711}}.
\newline\urlprefix\url{https://doi.org/10.1002/jmri.25711}

\bibitem{Zhang2017}
X.~Y. Zhang, Z.~L. Yang, G.~M. Lu, G.~F. Yang, L.~J. Zhang,
  \href{https://doi.org/10.3389/fnmol.2017.00343}{{PET/MR} imaging: {New}
  frontier in {Alzheimer's} disease and other dementias}, Front. Mol. Neurosci.
  10 (2017) 343.
\newblock \href {https://doi.org/10.3389/fnmol.2017.00343}
  {\path{doi:10.3389/fnmol.2017.00343}}.
\newline\urlprefix\url{https://doi.org/10.3389/fnmol.2017.00343}

\bibitem{Cecchin2017}
D.~Cecchin, H.~Barthel, D.~Poggiali, A.~Cagnin, S.~Tiepolt, P.~Zucchetta,
  P.~Turco, P.~Gallo, A.~C. Frigo, O.~Sabri, F.~Bui,
  \href{https://doi.org/10.1007/s00259-017-3750-0}{A new integrated dual
  time-point amyloid {PET/MRI} data analysis method}, Eur J Nucl Med Mol
  Imaging 44~(12) (2017) 2060--2072.
\newblock \href {https://doi.org/10.1007/s00259-017-3750-0}
  {\path{doi:10.1007/s00259-017-3750-0}}.
\newline\urlprefix\url{https://doi.org/10.1007/s00259-017-3750-0}

\bibitem{Tustison2014}
N.~J. Tustison, B.~B. Avants, P.~A. Cook, J.~Kim, J.~Whyte, J.~C. Gee, J.~R.
  Stone, \href{https://doi.org/10.1002/hbm.22211}{Logical circularity in
  voxel-based analysis: {Normalization} strategy may induce statistical bias},
  Hum. Brain Mapp 35~(3) (2012) 745--759.
\newblock \href {https://doi.org/10.1002/hbm.22211}
  {\path{doi:10.1002/hbm.22211}}.
\newline\urlprefix\url{https://doi.org/10.1002/hbm.22211}

\bibitem{Dumitrescu2019}
Dumitrescu, Boiangiu, \href{https://doi.org/10.3390/computers8020030}{A study
  of image upsampling and downsampling filters}, Computers 8~(2) (2019) 30.
\newblock \href {https://doi.org/10.3390/computers8020030}
  {\path{doi:10.3390/computers8020030}}.
\newline\urlprefix\url{https://doi.org/10.3390/computers8020030}

\bibitem{pog_gibbs21}
D.~Poggiali, D.~Cecchin, C.~Campi, S.~De~Marchi, Oversampling errors in
  multimodal medical imaging are due to the gibbs effect, Mathematics 9~(12)
  (2021).
\newblock \href {https://doi.org/10.3390/math9121348}
  {\path{doi:10.3390/math9121348}}.

\bibitem{Delgado2014}
J.~Delgado, J.~C. Moure, Y.~Vives-Gilabert, M.~Delfino, A.~Espinosa,
  B.~G{\'{o}}mez-Ans{\'{o}}n,
  \href{https://doi.org/10.1007/s12021-013-9214-1}{Improving the execution
  performance of {FreeSurfer}}, Neuroinformatics 12~(3) (2014) 413--421.
\newblock \href {https://doi.org/10.1007/s12021-013-9214-1}
  {\path{doi:10.1007/s12021-013-9214-1}}.
\newline\urlprefix\url{https://doi.org/10.1007/s12021-013-9214-1}

\bibitem{DeMarchi2020}
S.~De~Marchi, F.~Marchetti, E.~Perracchione, D.~Poggiali,
  \href{https://doi.org/10.1016/j.cam.2019.112347}{Polynomial interpolation via
  mapped bases without resampling}, J. Comput. Appl. Math. 364 (2020) 112347.
\newblock \href {https://doi.org/10.1016/j.cam.2019.112347}
  {\path{doi:10.1016/j.cam.2019.112347}}.
\newline\urlprefix\url{https://doi.org/10.1016/j.cam.2019.112347}

\bibitem{DeMarchi2021}
S.~De~Marchi, F.~Marchetti, E.~Perracchione, D.~Poggiali,
  \href{https://doi.org/10.1016/j.amc.2020.125628}{Multivariate approximation
  at fake nodes}, Appl. Math. Comput. 391 (2021) 125628.
\newblock \href {https://doi.org/10.1016/j.amc.2020.125628}
  {\path{doi:10.1016/j.amc.2020.125628}}.
\newline\urlprefix\url{https://doi.org/10.1016/j.amc.2020.125628}

\bibitem{burger_burge_2009}
W.~Burger, M.~J. Burge, Principles of digital image processing: core
  algorithms, Springer: Berlin/Heidelberg, Germany, 2009.

\bibitem{Getreuer2011}
P.~Getreuer, \href{https://doi.org/10.5201/ipol.2011.g_lmii}{Linear methods for
  image interpolation}, Image Processing On Line 1 (2011) 238--259.
\newblock \href {https://doi.org/10.5201/ipol.2011.g_lmii}
  {\path{doi:10.5201/ipol.2011.g_lmii}}.
\newline\urlprefix\url{https://doi.org/10.5201/ipol.2011.g_lmii}

\bibitem{Bugajewski2015}
D.~Bugajewski, J.~Gulgowski, P.~Kasprzak,
  \href{https://doi.org/10.1007/s10231-015-0526-7}{On continuity and
  compactness of some nonlinear operators in the spaces of functions of bounded
  variation}, Annali di Matematica Pura ed Applicata (1923 -) 195~(5) (2015)
  1513--1530.
\newblock \href {https://doi.org/10.1007/s10231-015-0526-7}
  {\path{doi:10.1007/s10231-015-0526-7}}.
\newline\urlprefix\url{https://doi.org/10.1007/s10231-015-0526-7}

\bibitem{Aronszajn1956}
N.~Aronszajn, P.~Panitchpakdi,
  \href{https://doi.org/10.2140/pjm.1956.6.405}{Extension of uniformly
  continuous transformations and hyperconvex metric spaces}, Pacific Journal of
  Mathematics 6~(3) (1956) 405--439.
\newblock \href {https://doi.org/10.2140/pjm.1956.6.405}
  {\path{doi:10.2140/pjm.1956.6.405}}.
\newline\urlprefix\url{https://doi.org/10.2140/pjm.1956.6.405}

\bibitem{Costarelli2018}
D.~Costarelli, G.~Vinti,
  \href{https://doi.org/10.1080/00036811.2018.1466277}{Quantitative estimates
  involving k-functionals for neural network-type operators}, Applicable
  Analysis 98~(15) (2018) 2639--2647.
\newblock \href {https://doi.org/10.1080/00036811.2018.1466277}
  {\path{doi:10.1080/00036811.2018.1466277}}.
\newline\urlprefix\url{https://doi.org/10.1080/00036811.2018.1466277}

\bibitem{Costarelli2019}
D.~Costarelli, G.~Vinti, \href{https://doi.org/10.33205/cma.484500}{A
  quantitative estimate for the sampling kantorovich series in terms of the
  modulus of continuity in orlicz spaces}, Constructive Mathematical Analysis
  (2019) 8--14\href {https://doi.org/10.33205/cma.484500}
  {\path{doi:10.33205/cma.484500}}.
\newline\urlprefix\url{https://doi.org/10.33205/cma.484500}

\bibitem{Jerri1998}
A.~Jerri, \href{https://doi.org/10.1002/zamm.200590016}{The {Gibbs} phenomenon
  in {Fourier} analysis, splines, and wavelet approximations}, Z. angew. Math.
  Mech. 85~(3) (2005) 224--224.
\newblock \href {https://doi.org/10.1002/zamm.200590016}
  {\path{doi:10.1002/zamm.200590016}}.
\newline\urlprefix\url{https://doi.org/10.1002/zamm.200590016}

\bibitem{Fornberg2011}
B.~Fornberg, N.~Flyer, The gibbs phenomenon for radial basis functions, The
  Gibbs Phenomenon in Various Representations and Applications, Sampling
  Publishing (2006).

\bibitem{Lehmann1999}
T.~Lehmann, C.~Gonner, K.~Spitzer,
  \href{https://doi.org/10.1109/42.816070}{Survey: {Interpolation} methods in
  medical image processing}, IEEE Trans. Med. Imaging 18~(11) (1999)
  1049--1075.
\newblock \href {https://doi.org/10.1109/42.816070}
  {\path{doi:10.1109/42.816070}}.
\newline\urlprefix\url{https://doi.org/10.1109/42.816070}

\bibitem{chloa20}
J.~F. Chhoa, {An Adaptive Approach to {Gibbs}' Phenomenon}, Master's thesis,
  The University of Southern Mississippi, Hattiesburg, Mississippi (2020).

\bibitem{DeMarchi2020a}
S.~De~Marchi, W.~Erb, E.~Francomano, F.~Marchetti, E.~Perracchione,
  D.~Poggiali, \href{https://doi.org/10.1109/melecon48756.2020.9140583}{Fake
  nodes approximation for magnetic particle imaging}, in: 2020 IEEE 20th
  Mediterranean Electrotechnical Conference (MELECON) Palermo, Italy, 16–18
  June 2020, IEEE, 2020, pp. 434--438.
\newblock \href {https://doi.org/10.1109/melecon48756.2020.9140583}
  {\path{doi:10.1109/melecon48756.2020.9140583}}.
\newline\urlprefix\url{https://doi.org/10.1109/melecon48756.2020.9140583}

\bibitem{graspa}
S.~D. Marchi, G.~Elefante, F.~Marchetti,
  \href{https://doi.org/10.1007/s40314-021-01688-z}{Stable discontinuous mapped
  bases: the gibbs{\textendash}runge-avoiding stable polynomial approximation
  ({GRASPA}) method}, Computational and Applied Mathematics 40~(8) (Nov. 2021).
\newblock \href {https://doi.org/10.1007/s40314-021-01688-z}
  {\path{doi:10.1007/s40314-021-01688-z}}.
\newline\urlprefix\url{https://doi.org/10.1007/s40314-021-01688-z}

\bibitem{betagamma}
S.~D. Marchi, G.~Elefante, F.~Marchetti,
  \href{https://doi.org/10.1016/j.jat.2021.105634}{On ($\beta$,
  $\gamma$)-chebyshev functions and points of the interval}, Journal of
  Approximation Theory 271 (2021) 105634.
\newblock \href {https://doi.org/10.1016/j.jat.2021.105634}
  {\path{doi:10.1016/j.jat.2021.105634}}.
\newline\urlprefix\url{https://doi.org/10.1016/j.jat.2021.105634}

\bibitem{Pham2000}
D.~L. Pham, C.~Xu, J.~L. Prince,
  \href{https://doi.org/10.1146/annurev.bioeng.2.1.315}{Current methods in
  medical image segmentation}, Annu. Rev. Biomed. Eng. 2~(1) (2000) 315--337.
\newblock \href {https://doi.org/10.1146/annurev.bioeng.2.1.315}
  {\path{doi:10.1146/annurev.bioeng.2.1.315}}.
\newline\urlprefix\url{https://doi.org/10.1146/annurev.bioeng.2.1.315}

\bibitem{Shepp1974}
L.~Shepp, B.~Logan, \href{https://doi.org/10.1109/tns.1974.6499235}{The
  {Fourier} reconstruction of a head section}, IEEE Trans. Nucl. Sci. 21~(3)
  (1974) 21--43.
\newblock \href {https://doi.org/10.1109/tns.1974.6499235}
  {\path{doi:10.1109/tns.1974.6499235}}.
\newline\urlprefix\url{https://doi.org/10.1109/tns.1974.6499235}

\bibitem{AVANTS2008}
B.~Avants, C.~Epstein, M.~Grossman, J.~Gee,
  \href{https://doi.org/10.1016/j.media.2007.06.004}{Symmetric diffeomorphic
  image registration with cross-correlation: {Evaluating} automated labeling of
  elderly and neurodegenerative brain}, Med. Image Anal. 12~(1) (2008) 26--41.
\newblock \href {https://doi.org/10.1016/j.media.2007.06.004}
  {\path{doi:10.1016/j.media.2007.06.004}}.
\newline\urlprefix\url{https://doi.org/10.1016/j.media.2007.06.004}

\bibitem{Avants2011}
B.~B. Avants, N.~J. Tustison, G.~Song, P.~A. Cook, A.~Klein, J.~C. Gee,
  \href{https://doi.org/10.1016/j.neuroimage.2010.09.025}{A reproducible
  evaluation of {ANTs} similarity metric performance in brain image
  registration}, Neuroimage 54~(3) (2011) 2033--2044.
\newblock \href {https://doi.org/10.1016/j.neuroimage.2010.09.025}
  {\path{doi:10.1016/j.neuroimage.2010.09.025}}.
\newline\urlprefix\url{https://doi.org/10.1016/j.neuroimage.2010.09.025}

\bibitem{itk_2014}
M.~Mccormick, X.~Liu, J.~Jomier, C.~Marion, L.~Ibanez, Itk: enabling
  reproducible research and open science, Frontiers in Neuroinformatics 8
  (2014).
\newblock \href {https://doi.org/10.3389/fninf.2014.00013}
  {\path{doi:10.3389/fninf.2014.00013}}.

\bibitem{ResizeRight}
A.~Shocher, Resizeright, \url{https://github.com/assafshocher/ResizeRight}
  (2018).

\bibitem{Brett2020}
M.~Brett, C.~J. Markiewicz, M.~Hanke, M.-A. C{\^{o}}t{\'{e}}, B.~Cipollini,
  P.~McCarthy, D.~Jarecka, C.~P. Cheng, Y.~O. Halchenko, M.~Cottaar, E.~Larson,
  S.~Ghosh, D.~Wassermann, S.~Gerhard, G.~R. Lee, H.-T. Wang, E.~Kastman,
  J.~Kaczmarzyk, R.~Guidotti, O.~Duek, et~al.,
  \href{https://zenodo.org/record/4295521}{nipy/nibabel: {3.2.1}}, Zenodo (nov
  2020).
\newblock \href {https://doi.org/10.5281/ZENODO.4295521}
  {\path{doi:10.5281/ZENODO.4295521}}.
\newline\urlprefix\url{https://zenodo.org/record/4295521}

\bibitem{rider_phantom}
P.~Muzi, M.~Wanner, P.~Kinahan,
  \href{https://wiki.cancerimagingarchive.net/x/bIRXAQ}{Data from rider phantom
  pet-ct} (2015).
\newblock \href {https://doi.org/10.7937/K9/TCIA.2015.8WG2KN4W}
  {\path{doi:10.7937/K9/TCIA.2015.8WG2KN4W}}.
\newline\urlprefix\url{https://wiki.cancerimagingarchive.net/x/bIRXAQ}

\bibitem{Clark2013}
K.~Clark, B.~Vendt, K.~Smith, J.~Freymann, J.~Kirby, P.~Koppel, S.~Moore,
  S.~Phillips, D.~Maffitt, M.~Pringle, L.~Tarbox, F.~Prior,
  \href{https://doi.org/10.1007/s10278-013-9622-7}{The cancer imaging archive
  ({TCIA}): Maintaining and operating a public information repository}, Journal
  of Digital Imaging 26~(6) (2013) 1045--1057.
\newblock \href {https://doi.org/10.1007/s10278-013-9622-7}
  {\path{doi:10.1007/s10278-013-9622-7}}.
\newline\urlprefix\url{https://doi.org/10.1007/s10278-013-9622-7}

\bibitem{Jha2014}
A.~K. Jha, N.~C. Purandare, S.~Shah, A.~Agrawal, A.~D. Puranik, V.~Rangarajan,
  \href{https://doi.org/10.4103/0971-3026.134379}{{PET} reconstruction artifact
  can be minimized by using sinogram correction and filtered back-projection
  technique}, Indian Journal of Radiology and Imaging 24~(02) (2014) 103--106.
\newblock \href {https://doi.org/10.4103/0971-3026.134379}
  {\path{doi:10.4103/0971-3026.134379}}.
\newline\urlprefix\url{https://doi.org/10.4103/0971-3026.134379}

\bibitem{Muller1992}
H.~W. Müller-Gärtner, J.~M. Links, J.~L. Prince, R.~N. Bryan, E.~McVeigh,
  J.~P. Leal, C.~Davatzikos, J.~J. Frost,
  \href{https://doi.org/10.1038/jcbfm.1992.81}{Measurement of radiotracer
  concentration in brain gray matter using positron emission tomography:
  {Mri-based} correction for partial volume effects}, J Cereb Blood Flow Metab
  12~(4) (1992) 571--583.
\newblock \href {https://doi.org/10.1038/jcbfm.1992.81}
  {\path{doi:10.1038/jcbfm.1992.81}}.
\newline\urlprefix\url{https://doi.org/10.1038/jcbfm.1992.81}

\bibitem{Cecchin2015}
D.~Cecchin, D.~Poggiali, L.~Riccardi, P.~Turco, F.~Bui, S.~De~Marchi,
  \href{https://doi.org/10.7717/peerj.722}{Analytical and experimental {FWHM}
  of a gamma camera: {Theoretical} and practical issues}, PeerJ 3~(2) (2015)
  e722.
\newblock \href {https://doi.org/10.7717/peerj.722}
  {\path{doi:10.7717/peerj.722}}.
\newline\urlprefix\url{https://doi.org/10.7717/peerj.722}

\bibitem{Allen2019}
M.~Allen, D.~Poggiali, K.~Whitaker, T.~R. Marshall, R.~A. Kievit,
  \href{https://doi.org/10.12688/wellcomeopenres.15191.1}{Raincloud plots: a
  multi-platform tool for robust data visualization}, Wellcome Open Research 4
  (2019) 63.
\newblock \href {https://doi.org/10.12688/wellcomeopenres.15191.1}
  {\path{doi:10.12688/wellcomeopenres.15191.1}}.
\newline\urlprefix\url{https://doi.org/10.12688/wellcomeopenres.15191.1}

\end{thebibliography}

\end{document}